\documentclass[11pt]{amsart}
\usepackage{amsfonts,epsfig}
\usepackage{latexsym}
\usepackage{amssymb}
\usepackage{amsmath}
\usepackage{amsthm}
\usepackage{graphics}
\usepackage[all]{xy}
\usepackage[T2A]{fontenc}
\usepackage{multirow}
\usepackage{hhline}
\usepackage{array, booktabs, ctable}
\usepackage{hyperref}
\usepackage{enumerate}

\newtheorem{defn}{Definition}[section]

\newtheorem{lemma}[defn]{Lemma}

\newtheorem{thm}[defn]{Theorem}

\newtheorem{cor}[defn]{Corollary}

\newtheorem{conj}[defn]{Conjecture}

\theoremstyle{definition}

\newtheorem*{ack}{Acknowledgements}
\newtheorem{remark}[defn]{Remark}

\newtheorem{example}[defn]{Example}

\newcommand{\Q}{\mathbb Q}
\newcommand{\Z}{\mathbb Z}

\newcommand{\F}{\mathbb F}

\newcommand{\Rats}{\mathbb{Q}}

\newcommand{\Ker}{\operatorname{Ker}}

\newcommand{\Gal}{\operatorname{Gal}}

\newcommand{\GQ}{\Gal(\overline{\Rats}/\Rats)}
\newcommand{\GK}{\Gal(\overline{K}/K)}

\newcommand{\GL}{\operatorname{GL}}
\newcommand{\PGL}{\operatorname{PGL}}

\begin{document}

\bibliographystyle{plain}
\title[On the $p$-primary torsion subgroup of elliptic curves]{On the minimal degree of definition of $p$-primary torsion subgroups of elliptic curves}

\author{Enrique Gonz\'alez--Jim\'enez}
\address{Universidad Aut{\'o}noma de Madrid, Departamento de Matem{\'a}ticas, Madrid, Spain}
\email{enrique.gonzalez.jimenez@uam.es}
\urladdr{http://www.uam.es/enrique.gonzalez.jimenez}
\author{\'Alvaro Lozano-Robledo}
\address{University of Connecticut, Department of Mathematics, Storrs, CT 06269, USA}
\email{alvaro.lozano-robledo@uconn.edu} 
\urladdr{http://alozano.clas.uconn.edu}

\thanks{The first author was partially  supported by the grant MTM2012--35849.}

\begin{abstract} In this article, we study the minimal degree $[K(T):K]$ of a $p$-subgroup  $T\subseteq E(\overline{K})_\text{tors}$  for an elliptic curve $E/K$ defined over a number field $K$. Our results depend on the shape of the image of the $p$-adic Galois representation $\rho_{E,p^\infty}:\GK\to \GL(2,\Z_p)$. However, we are able to show that there are certain uniform bounds for the minimal degree of definition of $T$. When the results are applied to $K=\Q$ and $p=2$, we obtain a divisibility condition on the minimal degree of definition of any subgroup of $E[2^n]$ that is best possible.
\end{abstract}

\maketitle

\section{Introduction}

Let $E$ be an elliptic curve defined over $\Q$, let $p$ be a prime number, and let $N\geq 1$. Let $\overline{\Q}$ be a fixed algebraic closure of $\Q$, and let $E[p^N]$ be the subgroup of algebraic points on $E(\overline{\Q})$ that are torsion points of order dividing $p^N$. In other words, $E[p^N]$ is the kernel of the multiplication-by-$p^N$ map $[p^N]:E\to E$. Let $T\subseteq E[p^N]$ be a $p$-subgroup. The central object of study of this article is the field of definition of $T$, namely $\Q(T):=\Q(\{x(P),y(P): P=(x(P),y(P))\in T \})$. 

It is well-known that a torsion subgroup $E(\Q)_\text{tors}$ may contain points of prime-power order $8$, $9$, $5$, or $7$, but no points of order $16$, $27$, $25$, $49$, or $p>7$. This follows for example from a theorem of Mazur (Theorem \ref{thm-mazur} below). Similarly, a theorem of Kamienny, Kenku, and Momose  (Theorem \ref{thm-quadgroups}), shows that there are quadratic fields $K$ and elliptic curves $E/K$ such that $E(K)$ contains a point of order $16$. However, one cannot find points of order $32$ in $E(K)$, for a curve $E/K$ and a quadratic extension $K/\Q$. Points of order $16$ also may appear starting from an elliptic curve $E/\Q$ and considering $E(K)$, where $K/\Q$ is quadratic (see Theorem \ref{thm-najman1}). More generally, the second author has shown that if $E/\Q$ is an elliptic curve over $\Q$, and $R\in E(\overline{\mathbb{Q}})$ is a point of order $p=11$ or $p>13$, then $[\Q(P):\Q]\geq (p-1)/2$ (see \cite{lozano0}, Theorem 2.1). 

In this article, we fix a number field $K$ and we study the minimal degree $[K(T):K]$ of a subgroup  $T\subseteq E(\overline{K})_\text{tors}$ with $T\cong \Z/p^s\Z\oplus \Z/p^N\Z$ for an elliptic curve $E/K$ defined over $K$. Our results depend on the shape of the image of the $p$-adic Galois representation $\rho_{E,p^\infty}:\GK\to \GL(2,\Z_p)$. However, we are able to show that there are certain uniform bounds for the minimal degree of definition of $T$. 

\begin{thm}\label{thm-mainintro1}
Let $p$ be a prime, let $K$ be a number field, and let $E/K$ be an elliptic curve defined over $K$ without complex multiplication. Let $0\leq s\leq N$ be integers, and let $T_{s,N}\subseteq E(\overline{K})_\text{tors}$ with $T_{s,N}\cong \Z/p^s\Z\oplus \Z/p^N\Z$. Then: 
\begin{enumerate} 
\item There are positive integers $n=n(K,p)$, $g_{s,M}(K,p)$, and $m_{s,M}(K,p)$, for $0\leq s\leq n$ and $M=\min\{n,N\}$, that depend on $K$ and $p$ but not on the choice of $E/K$ or $T_{s,N}$, such that 
the degree $[K(T_{s,N}):K]$ is divisible by $g_{s,M}(K,p)\cdot \max\{1, p^{2N-2n}  \}$, and  $[K(T_{s,N}):K]\geq m_{s,M}(K,p)\cdot \max\{1, p^{2N-2n}  \}$ if $s<n$, and the degree is divisible by $g_{n,n}(K,p)\cdot p^{2N+2s-4n}$, and  $[K(T_{s,N}):K]\geq m_{n,n}(K,p)\cdot p^{2N+2s-4n}$ if $n\leq s \leq N$.
\item For a fixed $E/K$, and for all but finitely many primes $p$, we have
$$[K(T_{s,N}):K] = \begin{cases} (p^2-1)p^{2N-2} &,  \text{ if } s=0,\\
(p-1)(p^2-1)p^{2N+2s-3} &, \text{ if } s\geq 1. \end{cases}$$
\end{enumerate}
\end{thm}

The integer $n(K,p)$ that appears in Theorem \ref{thm-mainintro1} is the smallest integer such that the image of the $p$-adic Galois representation $\rho_{E,p^\infty}$ is completely defined modulo $p^{n(K,p)}$ for all elliptic curves $E/K$ without complex multiplication. The existence of $n(K,p)$ is shown in Theorem \ref{thm-arai}.

We apply the uniform results of Theorem \ref{thm-mainintro1} to the case of $K=\Q$. For instance, we show that the minimal degree of a point of order $32$ for elliptic curves over $\Q$ is $8$. In fact, we show the following uniform divisibility result about the degree of the extension $\Q(P)/\Q$, for a point of arbitrary order $2^N$, for $N\geq 4$.

\begin{thm}\label{thm-main1}
Let $E/\Q$ be an elliptic curve defined over $\Q$ without CM, and let $P\in E[2^N]$ be a point of exact order $2^N$, with $N\geq 4$. Then, the degree $[\Q(P):\Q]$ is divisible by $2^{2N-7}$. Moreover, this bound is best possible, in the sense that there is a one-parameter family $E_t$ of elliptic curves over $\Q$ such that, for each $t\in\Q$, there is a points $P_{t,N}\in E_t(\overline{\Q})$ of exact order $2^N$, such that 
$$[\Q(P_{t,N}):\Q]=2^{2N-7}.$$ 
\end{thm}

The family mentioned in the statement of Theorem \ref{thm-main1} is $\mathcal{X}_{235l}$, which parametrizes all elliptic curves over $\Q$ with $2$-adic image \texttt{X235l} in the notation of \cite{rouse} (see Section \ref{sec-examples}). One concrete member of the family is the curve with Cremona label \cite{cremonaweb} \texttt{210e1}, given in Weierstrass form by
$$E: y^2+xy=x^3+210x+900.$$
As a corollary of Theorem \ref{thm-main1}, we obtain the following bound on two-torsion points.
\begin{cor}
Let $E/\Q$ be an elliptic curve without CM, and let $F/\Q$ be an extension of degree $d\geq 1$. Then $E(F)$ can only contain points of order $2^N$ with
$N\leq (\log_2(d)+7)/2.$ More precisely, if $\nu_2$ is the usual $2$-adic valuation, then $E(F)$ can only contain points of order $2^N$ with $N\leq \lfloor\frac{\nu_2(d)+7}{2}\rfloor$.
\end{cor}

Exploiting recent work of Rouse and Zureick-Brown \cite{rouse} that classifies all possible $2$-adic images for elliptic curves over $\Q$, one can calculate explicitly the constants $g_{s,N}(\Q,2)$ of Theorem \ref{thm-mainintro1} and calculate the minimal degree of definition of various subgroups of $E[2^N]$, for an elliptic curve $E$ defined over $\Q$. As before, Mazur's theorem implies that the $2$-primary component of any torsion subgroup $E(\Q)_\text{tors}$ is isomorphic to a subgroup of  $\Z/2\Z\oplus \Z/8\Z$, while the theorems of Kenku, Kamienny, and Momose imply that the $2$-primary components defined over a quadratic field are isomorphic to subgroups of  $\Z/16\Z$, $\Z/2\Z\oplus \Z/8\Z$, or $\Z/4\Z\oplus \Z/4\Z$. Our second main theorem, gives the best possible (divisibility) bound for the degree of definition of any torsion subgroup $T\cong \Z/2^s\Z \oplus \Z/2^{N}\Z$, for any $0\leq s\leq N$, for an elliptic curve over $\Q$ without CM.

\begin{thm}\label{thm-main2}
Let $E/\Q$ be an elliptic curve without CM. Let $1\leq s \leq N$ be fixed integers, and let $T\subseteq E[2^N]$ be a subgroup isomorphic to $\Z/2^s\Z \oplus \Z/2^{N}\Z$. Then, $[\Q(T):\Q]$ is divisible by $2$ if $s=N=2$, and otherwise by $2^{2N+2s-8}$ if $N\geq 3$, 
unless $s\geq 4$ and $j(E)$ is one of the two values
$$-\frac{3\cdot 18249920^3}{17^{16}} \text{ or } -\frac{7\cdot 1723187806080^3}{79^{16}}$$
in which case $[\Q(T):\Q]$ is divisible by $3\cdot 2^{2N+2s-9}$. Moreover, this bound is best possible, in the sense that there are one-parameter families $E_{s,N}(t)$ of elliptic curves over $\Q$ such that, for each $s,N\geq 0$ and each $t\in\Q$, and subgroups $T_{s,N}\in E_{s,N}(t)(\overline{\Q})$ isomorphic to $\Z/2^s\Z\oplus \Z/2^{N}\Z$, such that 
$[\Q(T_{s,N}):\Q]$ is equal to the bound given above. 
\end{thm}

We remark here that the $2$-torsion subgroups that are not covered by Theorems \ref{thm-main1} and \ref{thm-main2}, namely those that correspond to pairs $(s,N)=(0,1)$, $(0,2)$, $(0,3)$, $(1,1)$, and $(1,2)$, are known to appear infinitely many times as defined over $\Q$, by Mazur's theorem (Theorem \ref{thm-mazur}). Also, it is worth pointing out that the two $j$-invariants that appear in the statement of Theorem \ref{thm-main2} are two of the $j$-invariants that appear in Theorem 1.1 and Table 1 of \cite{rouse}.

\begin{example}

 As a consequence of Corollary \ref{cor-Q}, we can calculate the first degree $d$ where a certain $2$-primary torsion structure appears for some elliptic curve $E/\Q$ without CM. We do this in Table \ref{exmind} for $d\leq 16$, i.e., for each $d\leq 16$, we list all the possible torsion structures $\Z/2^s\Z\oplus \Z/2^{N}\Z$, such that there exists an elliptic curve $E$ defined over $\Q$ with $T_{s,N}\subseteq E_{s,N}(\overline{\Q})$ isomorphic to $\Z/2^s\Z\oplus \Z/2^{N}\Z$ and $[\Q(T_{s,N}):\Q]=d$, and such that $d$ is smallest with this property.

\begin{table}[h!]
\renewcommand{\arraystretch}{1.2}
$$
\begin{array}{|c|c|c|c|c|}
\hline
\multicolumn{5}{|c|}{d}\\
\hline
 1 & 2 & 4  & 8 & 16 \\
\hline   \begin{array}{c}\Z/2\Z \\ \Z/4\Z \\ \Z/8\Z\\ \Z/2\Z \oplus\Z/2\Z \\ \Z/2\Z \oplus\Z/4\Z \\ \Z/2\Z \oplus \Z/8\Z\end{array} & \begin{array}{c}\Z/16\Z \\ \Z/4\Z \oplus \Z/4\Z\end{array} & 
\begin{array}{c} \Z/2\Z \oplus \Z/16\Z \\ \Z/4\Z \oplus \Z/8\Z\end{array}  
& \Z/32\Z & 
\begin{array}{c} \Z/2\Z \oplus \Z/32\Z \\ \Z/4\Z \oplus \Z/16\Z \\ \Z/8\Z \oplus \Z/8\Z \end{array} 
\\ \hline   \end{array} 
$$
\caption{$2$-primary torsion subgroups that appear in degree $d$ for the first time.}\label{exmind}
\end{table}

\end{example}

For $p>2$, the classification of all possible $p$-adic images of Galois representations associated to elliptic curves $E/\Q$ is not known. In fact, the classification of all possible mod-$p$ images is not known, since we do not know whether there are elliptic curves without CM such that the mod-$p$ image is contained in a normalizer of a non-split Cartan subgroup of $\GL(2,\Z/p\Z)$ when $p\geq 13$ (see the introduction of \cite{lozano0} for a discussion of this topic). However, Sutherland and Zywina (\cite{sutherland}, \cite{zywina1}) have a list of $63$ mod-$p$ images that do occur for non-CM curves $E/\Q$, and that would be the complete  list if the answer to Serre's uniformity question is positive. Using this list of images, we can show the following theorem about $p$-adic representations that are defined modulo $p$, i.e., for those $p$-adic images that are the full inverse image of their mod-$p$ image.

\begin{thm}\label{thm-mainintro3}
Let $E/\Q$ be an elliptic curve without CM, and let $p$ be a prime such that 
\begin{enumerate}
\item[(A)] the image $G_1$ of $\rho_{E,p}:\GQ\to \GL(2,\Z/p\Z)$ is not contained in the normalizer of a non-split Cartan subgroup.
\end{enumerate} In addition, let us assume that either (B) or (C) occurs, where
\begin{enumerate}
\item[(B)] $p$ is not in the set $S=\{2,3,5,7,11,13,17,37\}$, or
\item[(C)] if $p\in S$, we suppose that  the $p$-adic image $G$ of $\rho_{E,p^\infty}$ is defined modulo $p$, i.e., the image $G$ of $\rho_{E,p^\infty}$ is the full inverse image of $G_1=\rho_{E,p}(\GQ)$ under mod-$p$ reduction. 
\end{enumerate} 
Let $T=T_{s,N}\cong \Z/p^s\Z\oplus \Z/p^N\Z\subseteq E[p^N]$ be a subgroup. Then, 
\begin{enumerate} 
\item For a fixed $G_1=\rho_{E,p}(\GQ)$, the degree $[\Q(T):\Q]$ is divisible by $g_{0,1}(G_1)\cdot p^{2N-2}$, and  $[\Q(T):\Q]\geq m_{0,1}(G_1)\cdot p^{2N-2}$ if $s=0$, and $[\Q(T):\Q]$ is divisible by $g_{1,1}(G_1)\cdot p^{2N+2s-4}$, and  $[\Q(T):\Q]\geq m_{1,1}(G_1)\cdot p^{2N+2s-4}$ if $s\geq 1$, where the constants $g_{k,1}(G_1)$ and $m_{k,1}(G_1)$ are given in Tables \ref{exSZ11} and \ref{exSZ37} for $k=0,1$.
\item In general, $[\Q(T):\Q]$ is divisible by $g_{0,1}(\Q,p)\cdot p^{2N-2}$, and  $[\Q(T):\Q]\geq m_{0,1}(\Q,p)\cdot p^{2N-2}$ if $s=0$, and divisible by $g_{1,1}(\Q,p)\cdot p^{2N+2s-4}$, and $[\Q(T):\Q]\geq m_{1,1}(\Q,p)\cdot p^{2N+2s-4}$ if $s\geq 1$, where the constants $g_{k,1}(\Q,p)$ and $m_{k,1}(\Q,p)$ are given in Table \ref{exQp0} for $k=0,1$.
\end{enumerate}
\end{thm}

\begin{table}[h!]
\begin{tabular}{c|cc|cc}
\hline
$p$ & $g_{0,1}(\Q,p)$ & $m_{0,1}(\Q,p)$ & $g_{1,1}(\Q,p)$ & $m_{1,1}(\Q,p)$ \\
\hline
$2$ & 1&1&1&1\\
$3$ & 1&1&2&2\\
$5$ & 1&1&4&4\\
$7$ & 1&1&6&18\\
$11$ & 5&5&10&110\\
$13$ & 1&3&12&288\\
$17$ & 8&8&1088&1088\\
$37$ & 12&12&15984&15984\\
else & $p^2-1$&$p^2-1$&$(p-1)p(p^2-1)$&$(p-1)p(p^2-1)$\\
\hline
\end{tabular}
\caption{$g_{k,1}(\Q,p)$ and $m_{k,1}(\Q,p)$, for $k=0,1$.}\label{exQp0}
\end{table}

For example, let $E/\Q$ be the elliptic curve 
$$y^2 + xy + y = x^3 + x^2 - 3x + 1$$
and Cremona label \texttt{50b1}. This is an elliptic curve without complex multiplication, such that $B=\rho_{E,3}(\GQ)$ is conjugate to a full Borel subgroup of $\GL(2,\Z/3\Z)$. In the notation of \cite{sutherland}, the image is \texttt{3B}. Moreover, the $3$-adic image $\rho_{E,3^\infty}(\GQ)$ is the full inverse image of $B$ in $\GL(2,\Z_3)$, as we will see in Section \ref{sec-modpimage}. Thus, $ \rho_{E,3^\infty}$ is defined modulo $3$, so $E/\Q$ satisfies (A) and (C) of Theorem \ref{thm-mainintro3}. It is easy to see that $g_{0,1}(\texttt{3B})=m_{0,1}(\texttt{3B})=2$ and $g_{1,1}(\texttt{3B})=m_{1,1}(\texttt{3B})=12$. Therefore, for every $0\leq s \leq N$, and every subgroup $T_{s,N}\cong \Z/3^s\Z\oplus \Z/3^N\Z \subseteq E[3^N]$, we have that 
$[\Q(T_{s,N}):\Q]$ is divisible by
$$\begin{cases}
2^{2N-1} &, \text{ if } s=0,\\
3\cdot 2^{2N+2s-2} &, \text{ if } s\geq 1,
\end{cases}$$
and these divisibility bounds are best possible, in the sense that, for each pair of $s,N$, there is a choice of $T'_{s,N}\subseteq E[3^N]$ such that $[\Q(T'_{s,N}):\Q]$ is equal to the bound.

\begin{remark}
All our results in this article are for elliptic curves without complex multiplication. In the CM case, there are known divisibility bounds for the field of definition of a point of order $N$ (and for $p$-primary torsion structures when the field of definition does not contain the quadratic field of complex multiplication) given by Silverberg \cite{silverberg}, and Prasad and Yogananda \cite{prasad}. More generally, Bourdon, Clark, and Pollack \cite{bourdon}, have recently shown divisibility bounds for $p$-primary torsion structures, similar to those of our Theorem \ref{thm-mainintro3}. 
\end{remark}

The article is organized as follows. In Section \ref{sec-prior} we record previous results in the literature related to our work, which will be used in Section \ref{sec-proofs} to prove our main Theorem \ref{thm-bounds}. Then, we will deduce a number of corollaries, which include Theorems \ref{thm-mainintro1}, \ref{thm-main1}, and \ref{thm-main2} from the introduction. In particular, Theorem \ref{thm-mainintro1} follows from Corollaries \ref{cor-full} and \ref{cor-gsNgeneral}, while \ref{thm-main1}, and \ref{thm-main2} follow from Corollary \ref{cor-gsN} (and they are equivalent to Corollary \ref{cor-Q}). Theorem \ref{thm-mainintro3} will be shown at the very end of Section \ref{sec-proofs}. Finally, in Section \ref{sec-examples}, we show examples of families of elliptic curves that achieve the minimal degrees of definition.

\begin{ack}
We would like to thank Andrew Sutherland and David Zywina for several helpful references and remarks. The second author would like to thank the Universidad Aut\'onoma de Madrid, where much of this work was completed during a sabbatical visit, for its hospitality. Finally, we would like to thank the editors and the referee(s) for a very efficient and helpful editorial and review process.
\end{ack}

\section{Prior Results in the Literature}\label{sec-prior}

The first three results that we quote describe the possible torsion subgroups for elliptic curves over $\Q$ or quadratic fields.

\begin{thm}[Mazur, \cite{mazur1}, Theorem 2]\label{thm-mazur}
Let $E/\Q$ be an elliptic curve. Then
\[
E(\Q)_\text{tors}\simeq
\begin{cases}
\Z/M\Z &\text{with}\ 1\leq M\leq 10\ \text{or}\ M=12,\ \text{or}\\
\Z/2\Z\oplus \Z/2M\Z &\text{with}\ 1\leq M\leq 4.
\end{cases}
\]
\end{thm}

\begin{thm}[Kenku, Momose, \cite{kenku2}; Kamienny, \cite{kami}]\label{thm-quadgroups}
Let $K/\Q$ be a quadratic field and let $E/K$ be an elliptic
curve. Then
\[
E(K)_\text{tors}\simeq
\begin{cases}
\Z/M\Z &\text{with}\ 1\leq M\leq 16\ \text{or}\ M=18,\ \text{or}\\
\Z/2\Z\oplus \Z/2M\Z &\text{with}\ 1\leq M\leq 6,\ \text{or}\\
\Z/3\Z \oplus \Z/3M\Z &\text{with}\ M=1\ \text{or}\ 2,\ \text{only if}\ K = \Q(\sqrt{-3}),\ \text{or}\\
\Z/4\Z \oplus \Z/4\Z &\text{only if}\ K = \Q(\sqrt{-1}).
\end{cases}
\]
\end{thm}

\begin{thm}[Najman, \cite{najman}, Theorem 2]
\label{thm-najman1} Let $E/\Q$ be an elliptic curve defined over $\Q$, and let $F$ be a quadratic number field. Then,
\[
E(F)_\text{tors}\simeq
\begin{cases}
\Z/M\Z &\text{with}\ 1\leq M\leq 10, \text{ or } M=12,15,\ \text{or}\ 16, \text{ or}\\
\Z/2\Z\oplus \Z/2M\Z &\text{with}\ 1\leq M\leq 6,\ \text{or}\\
\Z/3\Z \oplus \Z/3M\Z &\text{with}\ M=1\ \text{or}\ 2,\ \text{only if}\ F = \Q(\sqrt{-3}),\ \text{or}\\
\Z/4\Z \oplus \Z/4\Z &\text{only if}\ F = \Q(\sqrt{-1}).
\end{cases}
\]

\end{thm}

As we mentioned in the introduction, our results depend on the image of the $p$-adic Galois representation $\rho_{E,p^\infty}:\GK\to \GL(2,\Z_p)$ associated to the natural action of Galois on the Tate module $T_p(E)$ with respect to a fixed $\Z_p$-basis. In \cite{serre1}, Serre showed that the image of $\rho_{E,p^\infty}$ is as large as possible for all but finitely many prime numbers, as long as $E/K$ does not have complex multiplication.
\begin{thm}[Serre, \cite{serre1}]\label{thm-serre}
	Let $K$ be a number field, and let $E/K$ be an elliptic curve without complex multiplication. Then, $\rho_{E,p^\infty}(\GK)$ is an open subgroup of $\GL(2,\Z_p)$, and $\rho_{E,p^\infty}$ is surjective for all but finitely many primes. 
\end{thm} 

Serre's open image theorem implies that there is a number $n=n(E/K,p)$ such that  $1+p^nM_2(\Z_p)\subseteq \rho_{E,p^\infty}(\GK)$. The following result shows that $n(E/K,p)$ can be made independent of the curve.

\begin{thm}[Arai, \cite{arai}]\label{thm-arai} Let $K$ be a number field, and let $p$ be a prime. Then, there exists an integer $n=n(K,p)\geq 1$ depending on $K$ and $p$ such that for any elliptic curve $E$ over $K$ with no complex multiplication, we have $1+p^nM_2(\Z_p)\subseteq \rho_{E,p^\infty}(\GK)$. In other words, $\rho_{E,p^\infty}(\GK)$ is the full inverse image of $\rho_{E,p^n}(\GK)$ under reduction modulo $p^n$.
\end{thm} 

As a corollary of Arai's theorem, the image of $\rho_{E,p^\infty}$ is determined modulo $p^{n(K,p)}$, and so, the number of possible $p$-adic images (up to conjugation) is bounded above by the number of subgroups of $\GL(2,\Z/p^n\Z)$. Thus, we obtain that there are only finitely many possible $p$-adic images of $\rho_{E,p^\infty}$ over $K$ up to conjugation.

\begin{cor}\label{cor-arai}
Let $K$ be a number field, and let $p$ be a prime. Then, there is only a finite number $a(K,p)\geq 1$ of possibilities (up to conjugation) for the image of $\rho_{E,p^\infty}:\GK\to \GL(2,\Z_p)$, for any elliptic curve $E/K$ without complex multiplication. In other words, there are subgroups $G^i$ of $\GL(2,\Z_p)$, for $1\leq i \leq a(K,p)$, such that for any elliptic curve $E/K$ there is a number $j$ such that  $\rho_{E,p^\infty}(\GK)$ is a conjugate of $G^j$ in $\GL(2,\Z_p)$.
\end{cor} 

Rouse and Zureick-Brown have classified all the possible $2$-adic images of $\rho_{E,2}:\GQ\to \GL(2,\Z_2)$, and have shown that $n(\Q,2)=5$ and $a(\Q,2)=1208$, with notation as in Arai's theorem and its corollary. 

\begin{thm}[Rouse, Zureick-Brown, \cite{rouse}]\label{thm-rzb} Let $E$ be an elliptic curve over $\Q$ without complex multiplication. Then, there are exactly $1208$ possibilities for the $2$-adic image $\rho_{E,2^\infty}(\GQ)$, up to conjugacy in $\GL(2,\Z_2)$. Moreover:
\begin{enumerate}
\item The index of $\rho_{E,2^\infty}(\Gal(\overline{\Q}/\Q))$ in $\GL(2,\Z_2)$ divides $64$ or $96$.
\item The image $\rho_{E,2^\infty}(\GQ)$ is the full inverse image of $\rho_{E,2^5}(\GQ)$ under reduction modulo $2^5$.
\end{enumerate} 
\end{thm}

\begin{remark}\label{rem-rzb}
The $1208$ distinct possibilities for $2$-adic images that are found in \cite{rouse} are described in a few text files that can be found on the website listed in the references of this article. Each image has a label \texttt{Xk} or \texttt{Xkt} where \texttt{k} is a number and \texttt{t} is a letter (e.g., \texttt{X2} or \texttt{X58i}). In particular, the files  \texttt{curvelist1.txt} and \texttt{curvelist2.txt} are lists of examples of elliptic curves with each type of image, and the files \texttt{gl2data.gz} and \texttt{gl2finedata.gz} contain the descriptions of each image. The curves with each type of image come in $1$-parameter families which are given in the file \texttt{finemodels.tar.gz}. See the article and website \cite{rouse} for more info on how to interpret the files and notations. In addition, the website \cite{rouse} contains links to individual websites with data about each $2$-adic image. For instance, \href{http://users.wfu.edu/rouseja/2adic/X441.html}{http://users.wfu.edu/rouseja/2adic/X441.html} is the site for the image \texttt{X441}.
\end{remark}

For $p>2$, we know that the image of $\rho_{E,p}:\GQ\to\GL(E[p])$ is contained in one of the maximal subgroups of $\GL(E[p])\cong \GL(2,\Z/p\Z)$. The best results known are summarized in the following result.

\begin{thm}[Serre, \cite{serre1}, \S 2; \cite{serre2}, Lemme 18;  Mazur, \cite{mazur1}; Bilu, Parent, Rebolledo \cite{bilu}, \cite{bilu2}]\label{thm-serre2}
Let $E/\Q$ be an elliptic curve without CM. Let $G$ be the image of $\rho_{E,p}$, and suppose $G\neq \GL(E[p])$. Then one of the following possibilities holds:
\begin{enumerate}
\item $G$ is contained in a Borel subgroup of $\GL(E[p])$, and $p=2,3,5,7,11,13,17$, or $37$; or
\item The projective image of $G$ in $\PGL(E[p])$ is isomorphic to $A_4$, $S_4$ or $A_5$, where $S_n$ is the symmetric group and $A_n$ the alternating group, and $p\leq 13$; or
\item $G$ is contained in the normalizer of a split Cartan subgroup of $\GL(E[p])$ and $p\leq 13$, with $p\neq 11$; or 
\item $G$ is contained in the normalizer of a non-split Cartan subgroup of $\GL(E[p])$. 
\end{enumerate}
\end{thm}

Sutherland has computed the mod-$p$ image of all the non-CM elliptic curves in Cremona's tables and the Stein-Watkins database, some 140 million curves with conductors ranging up to $10^{12}$, and Zywina has described all known (and conjecturally all) proper subgroups of $\GL(2,\Z/p\Z)$ that occur as the image of $\rho_{E,p}$.

\begin{conj}[Sutherland, \cite{sutherland}; Zywina, \cite{zywina1}]\label{conj-zyw} Let $E/\Q$ be an elliptic curve without CM, and let $p$ be a prime. Then, there is a set $S_p$ formed by $s_p=|S_p|$ isomorphism types of subgroups of $\GL(2,\F_p)$, where 
	\begin{center}
		\begin{tabular}{c|ccccccccc}
			$p$ & $2$ & $3$ & $5$ & $7$ & $11$ & $13$ & $17$ & $37$ & else\\
			\hline
			$s_p$ & $3$ & $7$ & $15$ & $16$ & $7$ & $11$ & $2$ & $2$ & $0$,
		\end{tabular}
	\end{center} 
	such that if $G$ is the image of $\rho_{E,p}$, then $G$ is conjugate to one of the subgroups in $S$, or $G\cong \GL(2,\F_p)$.
\end{conj}
The list of images in the sets $S_p$ appears in our Tables \ref{exSZ11} and \ref{exSZ37} and are described in \cite{zywina1}, and Tables 3 and 4 of \cite{sutherland}.

\section{Proofs}\label{sec-proofs}

We begin by calculating a general formula for $[K(T):K]$ in terms of the sizes of subgroups of the Galois group of $\Q(E[p^N])/\Q$.

\begin{thm}\label{thm-bounds}
Let $p$ be a prime number, let $K$ be a number field, let $E/K$ be an elliptic curve without CM, let $G=\rho_{E,p^\infty}(\GK)$, and suppose that $G$ is defined at level $p^d$, for some $d\geq 1$ (i.e., $G$ is the full inverse image of $G_d\equiv G\bmod p^d$ under reduction modulo $p^d$). Let $0\leq s\leq N$ be fixed integers, and let $T\subseteq E[p^N]$ be a subgroup isomorphic to $\Z/p^s\Z \oplus \Z/p^{N}\Z$, and write $H_T$ for the subgroup of $G_N\cong \Gal(K(E[p^N])/K)$ that fixes $K(T)$, so that if $s>0$ we have 
$$H_T=\left\{ \begin{pmatrix}
1 & 0 \\
0 & 1 \\
\end{pmatrix} + p^s\cdot \begin{pmatrix}
0 & a \\
0 & b \\
\end{pmatrix}: a,b \in \Z/p^{N}\Z  \right\}\cap G_N \subseteq \GL(2,\Z/p^{N}\Z),$$
where the chosen basis of $E[p^N]$ is $\{P,Q \}$ and $P\in T$ is a point of order $p^N$, and  if $s=0$, then  
$$H_T=\left\{ \begin{pmatrix}
1 & a \\
0 & b \\
\end{pmatrix}: a\in \Z/p^{N}\Z,\ b\in (\Z/p^N\Z)^\times  \right\}\cap G_N \subseteq \GL(2,\Z/p^{N}\Z).$$ Then, the index $[K(T):K]$ is computed as follows:
\begin{enumerate}[(i)]
\item If $0<s\leq N\leq d$, then $[K(T):K]=|G_N|/|H_T|$  where $G_N \equiv G\equiv G_d \bmod p^N$. 

\item If $s\leq d\leq N$, then 
$$[K(T):K]=\frac{|G_N|}{|H_T|}=\frac{|G_d|}{|H_{T_d}|}\cdot p^{2(N-d)},$$
where $H_{T_d}$ is the subgroup of $G_d$ that fixes $T_d=T\cap E[p^d]$.
\item If $d\leq s\leq N$, then $$[K(T):K]=\frac{|G_N|}{|H_T|}=|G_d|\cdot p^{2N+2s-4d}.$$

\end{enumerate}
\end{thm}
\begin{proof} 
Let $p$ be a prime number, let $E/K$ be an elliptic curve without CM, and let $\rho_{E,p^\infty}$ be the Galois representation $\GK\to \GL(2,\Z_p)$  associated to the action of Galois on the Tate module $T_p(E)$ after fixing a $\Z_p$-basis $\{P,Q\}$ of $T_p(E)$, and suppose that $G$ is defined at level $p^d$, for some $d\geq 1$ (i.e., $G$ is the full inverse image of $G_d\equiv G\bmod p^d$ under reduction modulo $p^d$). For each $m\geq 1$, let $\rho_{E,p^m}:\GK \to \GL(2,\Z/p^m\Z)$ be the Galois representations obtained as reduction of $\rho_{E,p}$ modulo $p^m$, and let $G_m$ be the image of $\rho_{E,p^m}$. Then, $G_m\cong \Gal(K(E[p^m])/K)$. 

Let $s$, $N$, and $T$ be as in the statement of the theorem. Let $H_T$ be the subgroup of $G_N\cong \Gal(K(E[p^N])/K)$ that fixes $K(T)$. In particular, 
$$[K(T):K]=\frac{|\Gal(K(E[p^N])/K)|}{|H_T|}=\frac{|G_N|}{|H_T|}.$$ Since $T$ contains a point $P_N$ of order $p^N$, we can choose $Q_N\in E[p^N]$ that forms a $\Z/p^N\Z$-basis $\{P_N,Q_N\}$ of $E[p^N]$. With respect to this basis, the subgroup $H_T$ fixing $T$ must be the subgroup of matrices in $G_N\subseteq \GL(2,\Z/p^N\Z)$ that (a) fix $P_N$, and (b) reduce to the identity modulo $p^s$ (since $E[p^s]\cong \Z/p^s\Z \oplus \Z/p^{s}\Z$ is a subgroup of $T$). Hence, if $s>0$, the group $H_T$ must be given by
$$H_T=\left\{ \begin{pmatrix}
1 & 0 \\
0 & 1 \\
\end{pmatrix} + p^s\cdot \begin{pmatrix}
0 & a \\
0 & b \\
\end{pmatrix}: a,b \in \Z/p^{N}\Z  \right\}\cap G_N \subseteq \GL(2,\Z/p^{N}\Z),$$
and if $s=0$, then  
$$H_T=\left\{ \begin{pmatrix}
1 & a \\
0 & b \\
\end{pmatrix}: a\in \Z/p^{N}\Z,\ b\in (\Z/p^N\Z)^\times  \right\}\cap G_N \subseteq \GL(2,\Z/p^{N}\Z).$$
We distinguish three cases according to whether (i) $s\leq N\leq d$, or (ii) $s\leq d \leq N$, or (iii) $d\leq s\leq N$:

\begin{enumerate}[(i)]
\item Suppose  $s\leq N\leq d$. Then, $[K(T):K]=|G_N|/|H_T|$ can be calculated directly, where $G_N \equiv G_d \bmod p^N$, and $H_T$ is defined as before.
\item Suppose  $s\leq d\leq N$. Since $N\geq d$, the subgroup $G_N$ is the full inverse image of $G_d$. In particular, since $\Ker(\GL(2,\Z/p^{d+1}\Z)\to \GL(2,\Z/p^d\Z))$ is the subgroup $$\left\langle \begin{pmatrix}
1 & p^d \\
0 & 1 \\
\end{pmatrix}, \begin{pmatrix}
1 & 0 \\
p^d & 1 \\
\end{pmatrix}, \begin{pmatrix}
1+p^d & 0 \\
0 & 1 \\
\end{pmatrix}, \begin{pmatrix}
1 & 0 \\
0 & 1+p^d \\
\end{pmatrix} \right\rangle,$$
of order $p^4$, it follows that $\Ker(\GL(2,\Z/p^{N}\Z)\to \GL(2,\Z/p^d\Z))$ is a subgroup of order $p^{4(N-d)}$. Hence:
$$|G_N|=|\Gal(K(E[p^N])/K)| = |\Gal(K(E[p^d])/K)|\cdot p^{4(N-d)}= |G_d|\cdot p^{4(N-d)}.$$
Let $H_{T_d}$ be the subgroup of $G_d$ that fixes $T_d=T\cap E[p^d]\cong \Z/p^s\Z \oplus \Z/p^{d}\Z$, where a fixed point of order $p^d$ is $P_d=p^{N-d}P_N$. If we write $Q_d=p^{N-d}Q_N$, and we assume $s>0$ for now, then $H_{T_d}$ is the subgroup 
$$H_{T_d}=\left\{ \begin{pmatrix}
1 & 0 \\
0 & 1 \\
\end{pmatrix} + p^s\cdot \begin{pmatrix}
0 & a \\
0 & b \\
\end{pmatrix}: a,b \in \Z/p^{d}\Z  \right\}\cap G_d \subseteq \GL(2,\Z/p^{d}\Z),$$
and if $s=0$, then 
$$H_{T_d}=\left\{ \begin{pmatrix}
1 & a \\
0 & b \\
\end{pmatrix}: a\in \Z/p^{d}\Z,\ b\in (\Z/p^d\Z)^\times  \right\}\cap G_d \subseteq \GL(2,\Z/p^{d}\Z).$$
In either case, $H_{T_d}\equiv H_T \bmod p^d$, and $|H_{T}|/|H_{T_d}|=p^{2(N-d)}$. Hence,
$$[K(T):K]=\frac{|G_N|}{|H_T|}=\frac{|G_d|\cdot p^{4(N-d)}}{|H_{T_d}|\cdot p^{2(N-d)}}=\frac{|G_d|}{|H_{T_d}|}\cdot p^{2N-2d}.$$

We remark here that this formula is the same for any subgroup $T\subseteq E[p^N]$ such that $T_d=T\cap E[p^d]$. Thus, it only depends on the size of $|G_d|/|H_{T_d}|$, for each possible $T_d\cong \Z/p^s\Z\oplus \Z/p^d\Z$. 
\item Suppose that $d\leq s\leq N$. As before, $G_N$ is the full inverse image of $G_d$, and so $|G_N|=|G_d|\cdot p^{4(N-d)}$. 
In this case, $d\leq s$, and so, by the definition of $H_T\subseteq G_N$, every $M\in H_T$ reduces to the identity modulo $s$ and thus modulo $d$. Since $N\geq s \geq d\geq 1$, it follows that 
$$H_T=\left\{ \begin{pmatrix}
1 & 0 \\
0 & 1 \\
\end{pmatrix} + p^s\cdot \begin{pmatrix}
0 & a \\
0 & b \\
\end{pmatrix}: a,b \in \Z/p^{N}\Z  \right\}$$
and, therefore,
$$|H_T|=p^{2(N-s)},$$
and
$$[K(T):K]=\frac{|G_N|}{|H_T|}=\frac{|G_d|\cdot p^{4(N-d)}}{p^{2(N-s)}}=|G_d|\cdot p^{2N+2s-4d}.$$

\end{enumerate}
\end{proof} 

In the following corollary, we apply Theorem \ref{thm-bounds} to the case of full $\GL(2,\Z_p)$ image. We remind the reader that the order 
$$|\GL(2,\Z/N\Z)|=\varphi(N)\cdot N^3 \prod_{p|N}(1-1/p^2)$$
for any $N\geq 1$. In particular, 
$$|\GL(2,\Z/p^N\Z)|=(p-1)p^{N-1}\cdot p^{3N} \cdot (1-1/p^2)=(p-1)(p^2-1)p^{4N-3}.$$

\begin{cor}\label{cor-full}
Let $K$ be a number field, and let $E/K$ be an elliptic curve. Then: 
\begin{enumerate}
	\item Suppose $G=\rho_{E,p^\infty}(\GK)\cong \GL(2,\Z_p)$ for some prime $p$. Let $0\leq s\leq N$ be fixed integers, let $T_{s,N}\subseteq E[p^N]$ be a subgroup isomorphic to $\Z/p^s\Z \oplus \Z/p^{N}\Z$. Then:
$$[K(T_{s,N}):K] = \begin{cases} (p^2-1)p^{2N-2} &,  \text{ if } s=0,\\
(p-1)(p^2-1)p^{2N+2s-3} &, \text{ if } s\geq 1. \end{cases}$$
 \item For a fixed elliptic curve $E/K$ without CM the degree $[K(T_{s,N}):K]$ is given by the formula in (1) for all but finitely many primes $p$.
\end{enumerate}  
\end{cor}
\begin{proof}
First, let $s=0$ and let $N\geq 1$ be arbitrary. Let $T=\langle P \rangle$ and let $\{P,Q\}$ be a basis of $E[p^N]$, and write $H_T$ for the subgroup of $G_N\cong \Gal(K(E[p^N])/K)$ that fixes $K(T)$. Then, the formulae of Theorem \ref{thm-bounds} says that $|H_T|=\varphi(p^N)\cdot p^N = (p-1)p^{N-1}p^N$, while $|G_N|=|\GL(2,\Z/p^N\Z)|=(p-1)(p^2-1)p^{4N-3}$. Thus,
$$[K(T):K]=[K(P):K]=\frac{|G_N|}{|H_T|}=\frac{(p-1)(p^2-1)p^{4N-3}}{(p-1)p^{2N-1}}=(p^2-1)p^{2N-2}.$$ 
Otherwise, $N\geq s\geq 1=d$, and we are in case (iii) of Theorem \ref{thm-bounds}. Hence,
$$[K(T):K]=|G_1|\cdot p^{2N+2s-4} =(p-1)(p^2-1)p\cdot p^{2N+2s-4}=(p-1)(p^2-1)p^{2N+2s-3}.$$ 
This shows (1). Now (2) is a direct consequence of Serre's open image theorem (Theorem \ref{thm-serre}).
\end{proof}

Corollary \ref{cor-arai} says that there is a finite number of possible $p$-adic images $\rho_{E,p^\infty}(\GK)$, up to conjugation, for any elliptic curve $E/K$. Thus, there are divisibility bounds  as in Theorem \ref{thm-bounds} that are uniform over all elliptic curves over $K$.

\begin{cor}\label{cor-gsNgeneral}
Let $p$ be a prime, let $K$ be a number field, and let $n(K,p)\geq 1$ be the number given by Theorem \ref{thm-arai}. Let $G^i$, for $i=1,\ldots, a(K,p)$, be the possible $p$-adic images for $\rho_{E,p^\infty}$, given by Corollary \ref{cor-arai}. Let $E/K$ be an elliptic curve without CM. Let $0\leq s \leq N$ be fixed integers, let $T\subseteq E[p^N]$ be any subgroup isomorphic to $\Z/p^s\Z \oplus \Z/p^{N}\Z$, and put $M=\min\{N,n(K,p)\}$. Then:
\begin{enumerate} 
\item  Suppose $\rho_{E,p^\infty}(\GK)=G^i$, for a fixed $1\leq i \leq a(K,p)$. Then, $[K(T):K]$ is divisible by $g_{s,M}(G^i)\cdot \max\{1, p^{2N-2n(K,p)} \}$ if $s<n(K,p)$ and by $g_{n(K,p),n(K,p)}(G^i)\cdot p^{2N+2s-4n(K,p)}$ if $s\geq n(K,p)$, where 
$$g_{s,M}(G^i)=\gcd\left(\left\{ \frac{|G^i_M|}{|H^i_{T'}|} : T'\cong \Z/p^s\Z\oplus \Z/p^{M}\Z\subseteq \langle P^i_M,Q^i_M\rangle \right\} \right),$$
the group $G^i_M$ is defined by $G^i_M\equiv G^i_{n(K,p)}\equiv G^i\bmod p^M$ as a subgroup of $\GL(2,\Z/p^M\Z)$, and $H^i_{T'}$ is the subgroup of $G^i_M$ that fixes $T'$.
\item More generally, the degree $[K(T):K]$ is divisible by $g_{s,M}(K,p)\cdot \max\{1, p^{2N-2n(K,p)} \}$ if $s<n(K,p)$ and by $g_{n(K,p),n(K,p)}(K,p)\cdot p^{2N+2s-4n(K,p)}$ if $s\geq n(K,p)$, where
$$g_{s,M}(K,p) = \gcd(\{g_{s,M}(G^i) : 1\leq i \leq a(K,p)\}).$$
\end{enumerate}
Finally, if we define $m_{s,N}(G^i)$ and $m_{s,N}(K,p)$ like $g_{s,N}$ replacing $\gcd$ by $\min$, then $[K(T):K]\geq $
$$ \begin{cases} m_{s,M}(G^i)\cdot \max\{1, p^{2N-2n(K,p)} \}\geq m_{s,M}(K,p)\cdot \max\{1, p^{2N-2n(K,p)} \} &, \text{ if }  s<n(K,p),  \text{ and } \\
  m_{n(K,p),n(K,p)}(G^i)\cdot p^{2N+2s-4n(K,p)}\geq m_{n(K,p),n(K,p)}(K,p)\cdot p^{2N+2s-4n(K,p)} 
&, \text{ if }  s\geq n(K,p).\end{cases}$$
\end{cor} 
\begin{proof} 
The result follows from the divisibility bounds of Theorem \ref{thm-bounds}. Clearly part (2) is a direct consequence of (1) and the fact that there are only $a(K,p)$ possible $p$-adic images (Corollary \ref{cor-arai}), so we will just prove (1). By Theorem \ref{thm-arai} for every elliptic curve $E/K$ without CM, the image $G=\rho_{E,p^\infty}(\GK)$ is defined at most modulo $p^{n(K,p)}$ so, for our purposes, we will define every representation exactly modulo $p^{n(K,p)}$, i.e., the value of $d$ in Theorem \ref{thm-bounds} will be always $d=n(K,p)$. Each image $G^i$ is then defined by $G^i_{d}\equiv G^i\bmod p^{d}$ as a subgroup of $\GL(2,\Z/p^{d}\Z)$, with respect to a basis $\{P^i_d,Q^i_d\}$ of $E^i[p^d]$, where $E^i$ is an elliptic curve with $G=\rho_{E,p^\infty}(\GK)=G^i$. We also fix a compatible basis $\{P^i_n,Q^i_n\}$ of $E^i[p^n]$ for each $n\geq 1$.

There are three possibilities according to the values of $s$, $N$, and $d=n(K,p)$. 
\begin{enumerate}[(i)]
\item Suppose $s\leq N\leq n(K,p)$. Then, Theorem \ref{thm-bounds}, part (i) shows that $[K(T):K]$ is divisible by the quantity $g_{s,N}(G^i)$, where 
$$g_{s,N}(G^i)=\gcd\left(\left\{ \frac{|G^i_N|}{|H^i_{T'}|} :  T'\cong \Z/p^s\Z\oplus \Z/p^{N}\Z\subseteq \langle P^i_N,Q^i_N\rangle \right\} \right),$$
where $H^i_{T'}$ is the subgroup of $G^i_N$ that fixes $T'$, and $G^i_N\equiv G^i_{n(K,p)}\equiv G^i \bmod p^N$.
\item Suppose $s\leq n(K,p)\leq N$. Then, Theorem \ref{thm-bounds}, part (ii) shows that $[K(T):K]$ is divisible by $g_{s,n(K,p)}(G^i)\cdot p^{2N-2n(K,p)}$, where  $$g_{s,n(K,p)}(G^i)=\gcd\left(\left\{ \frac{|G^i_{n(K,p)}|}{|H^i_{T'}|} :  T'\cong \Z/p^s\Z\oplus \Z/p^{n(K,p)}\Z\subseteq \langle P^i_{n(K,p)},Q^i_{n(K,p)}\rangle \right\}\right),$$
where $H^i_{T'}$ is the subgroup of $G^i_{n(K,p)}$ that fixes $T'$.
\item Suppose $n(K,p)\leq s\leq N$. Then, Theorem \ref{thm-bounds}, part (iii) shows that $[K(T):K]$ is divisible by $g_{n(K,p),n(K,p)}(G^i)\cdot p^{2N+2s-4n(K,p)}$, where
 $g_{n(K,p),n(K,p)}(G^i)=|G^i_{n(K,p)}|.$
\end{enumerate}
Hence, in all cases, $[K(T):K]$ is divisible by $g_{s,M}(G^i)\cdot \max\{1, p^{2N-2n(K,p)} \}$ if $s<n(K,p)$ and by $g_{n(K,p),n(K,p)}(G^i)\cdot p^{2N+2s-4n(K,p)}$ if $s\geq n(K,p)$, as claimed.
\end{proof}

By the results of \cite{rouse}, we know that $n(\Q,2)=5$ and $a(\Q,2)=1208$, so we can specialize the previous result to $p=2$ and $K=\Q$, as follows.

\begin{cor}\label{cor-gsN}
Let $E/\Q$ be an elliptic curve without CM. Let $0\leq s \leq N$ be fixed integers, let $T\subseteq E[2^N]$ be a subgroup isomorphic to $\Z/2^s\Z \oplus \Z/2^{N}\Z$, and put $M=\min\{N,5\}$. Then:
\begin{enumerate}
\item Suppose that $\rho_{E,2^\infty}(\GQ)=G^i$. Then, 
 $[\Q(T):\Q]$ is divisible by the number $g_{s,M}(G^i)\cdot \max\{1, 2^{2N-10} \}$ if $s<5$ and by $g_{5,5}(G^i)\cdot 2^{2N+2s-20}$ if $s\geq 5$, where the numbers 
$g_{s,M}(G^i)$ are defined as in Corollary \ref{cor-gsNgeneral}. 
\item More generally, $[\Q(T):\Q]$ is divisible by $g_{s,M}(\Q,2)\cdot \max\{1, 2^{2N-10} \}$ if $s<5$ and by the number $g_{5,5}(\Q,2)\cdot 2^{2N+2s-20}$ if $s\geq 5$, where $g_{s,M}(\Q,2)=\gcd(\{g_{s,M}(G^i): 1\leq i \leq 1208  \})$, and the values $g_{s,M}(\Q,2)$ are given by Table \ref{ex}.
\begin{table}[h!]
\renewcommand{\arraystretch}{1.3}
\begin{tabular}{r|c|ccccc| }
       \multicolumn{2}{c|}{\multirow{2}{*}{$g_{s,M}(\Q,2)$}}                     &  & & $M$ &  &    \\             
                             \cline{3-7}
             \multicolumn{2}{c|}{}              &  $1$       &     $2$     &    $3$    &     $ 4$      &     $5$        \\
                               \hline
                    & $0$          &  $1$       &     $1$     &    $1$    &     $ 2$      &     $2^3$     \\
   & $1$&  $1$       &     $1$     &    $1$    &      $2^2$      &     $2^4$      \\
    &$2$ &          &     $2$    &    $2^2$    &      $2^4$      &   $2^6$       \\
 $s$  & $3$&           &          &      $2^4$  &      $2^{6}$      &     $2^{8}$       \\
   &$4$ &           &            &          &      $2^{7}$      &     $2^{9}    $   \\
   &  $5$ &        &            &          &            &   $  2^{11}   $    \\  
   \hline \multicolumn{7}{c}{\phantom{2}}\\
\end{tabular}
\caption{$g_{s,M}(\Q,2)$, for $0\le s\le M\le 5$}\label{ex}
\end{table}
\end{enumerate}
Furthermore:
\begin{enumerate}[(a)]
\item The values $g_{s,M}(\Q,2)$ are best possible for $0\leq s \leq 3$, i.e., there is a type of representation $G^i$ such that $g_{s,M}(\Q,2)=g_{s,M}(G^i)=m_{s,M}(G^i)$ (with notation as in Cor. \ref{cor-gsNgeneral}).
\item For any $1\leq i\leq 1208$, and if $s\geq 4$, then $g_{s,M}(G^i)$ is divisible by $2\cdot g_{s,M}(\Q,2)$ or $3\cdot g_{s,M}(\Q,2)$, and both cases occur with equality for certain $2$-adic images.
\item Let $s\geq 4$. The image $G^i$ is such that $g_{s,M}(G^i)$ is divisible by $3\cdot g_{s,M}(\Q,2)$ but not by $2\cdot g_{s,M}(\Q,2)$ if and only if $G^i$ corresponds to the $2$-adic image with label {\rm  \texttt{X441}}. 
\item The minimal values $m_{s,M}(\Q,2)$ are equal to $g_{s,M}(\Q,2)$ for $0\leq s \leq 3$, and equal to $2\cdot g_{s,M}(\Q,2)$ for $s\geq 4$.
\end{enumerate} 
\end{cor} 
\begin{proof}
The first part is immediate from Corollary \ref{cor-gsNgeneral}. The values $g_{s,N}(G^i)$ and $g_{s,N}(\Q,2)$ have been calculated using Magma \cite{magma} for each possible image group $G^i$ as described by \cite{rouse}. In our calculations we have found all possible values of $|G^i_N|/|H^i_{T'}|$ for each $N\leq 5$, for each choice of $T'\subseteq E[2^N]$, and for all $1\leq i\leq a(\Q,2)=1208$.  The Magma scripts used in this computation can be found at the research website of either author. In particular, the file \texttt{2primary\_Ss.txt} contains the values of $g_{s,t}(G^i)$, for $0\leq s\leq 5$ and $1\leq t\leq 5$, for each of the possible $2$-adic images that occur for elliptic curves over $\Q$. In other words, in the file \texttt{2primary\_Ss.txt} the reader can find an analogue of Table \ref{ex2} for each of the $1208$ possible $2$-adic images for non-CM curves. We shall outline here the computation for a given image. 

For instance, say $G^i=G$ corresponds to the $2$-adic image \texttt{X235l} (as always, following the notation of \cite{rouse}), which is defined modulo $16$, and is generated in $\GL(2,\Z/16\Z)$ by
$$
G_4 = \left\langle
\begin{pmatrix}
 1 & 0\\
 1 & 1 
 \end{pmatrix} ,
 \begin{pmatrix}
  1 & 0\\
   12 & 1  \end{pmatrix} ,
 \begin{pmatrix}
 9 & 0\\
  0 & 1 
   \end{pmatrix} ,
 \begin{pmatrix}
 1 & 0\\
  14 & 1  \end{pmatrix} ,
 \begin{pmatrix}
 5 & 0\\
  0 & 1 
   \end{pmatrix} ,
 \begin{pmatrix}
 15 & 0\\
  0 & 1  
  \end{pmatrix} ,
 \begin{pmatrix}
 9 & 0\\
  8 & 9 
   \end{pmatrix} ,
 \begin{pmatrix}
 1& 0\\ 
 8 & 1 
  \end{pmatrix} 
\right\rangle.
$$
In our notation $G_N$ acts on $E[2^N]$ on the left, i.e., $M\in G_N$ acts on $R\in E[2^N]$ by $M\cdot R$, so our matrices are transposed from those that appear in \cite{rouse}. Then, the values $g_{s,M}(G)$ are given in Table \ref{ex2}. We shall work out two cases in detail: $(s,N)=(0,1)$ and $(s,N)=(1,5)$.

\begin{table}[h!]
\renewcommand{\arraystretch}{1.3}
\begin{tabular}{r|c|ccccc| }
       \multicolumn{2}{c|}{\multirow{2}{*}{$g_{s,M}(G)$}}                     &  & & $M$ &  &    \\             
                             \cline{3-7}
             \multicolumn{2}{c|}{}              &  $1$       &     $2$     &    $3$    &     $ 4$      &     $5$        \\
                               \hline
                    & $0$          &  $1$       &     $1$     &    $1$    &     $ 2$      &     $2^3$     \\
   & $1$&  $2$       &     $2$     &    $2$    &      $2^2$      &     $2^4$      \\
    &$2$ &          &     $2^3$    &    $2^3$    &      $2^4$      &   $2^6$       \\
 $s$  & $3$&           &          &      $2^5$  &      $2^{6}$      &     $2^{8}$       \\
   &$4$ &           &            &          &      $2^{8}$      &     $2^{10}    $   \\
   &  $5$ &        &            &          &            &   $  2^{12}   $    \\  
   \hline \multicolumn{7}{c}{\phantom{2}}\\
\end{tabular}
\caption{Values $g_{s,M}(G)$, for $0\le s\le M\le 5$, where $G$ is the image with label \texttt{X235l}.}\label{ex2}
\end{table}

Let $s=0$ and $N=1$. We first compute $G_1$ which is the reduction of $G$ modulo $2$, so it is generated by the reduction of every generator of $G_4$ modulo $2$. Thus,
$$G_1 = \left\langle
\begin{pmatrix}
 1 & 0\\
 1 & 1 
 \end{pmatrix} \right\rangle \subset \GL(2,\Z/2\Z).
 $$
As $T'$ runs over the three subgroups $T_1'=\langle(1,0)\rangle$, $T_2'=\langle(0,1)\rangle$, and $T_3'=\langle(1,1)\rangle$ of order $2$ of $E[2]$ (where the points are given in the same coordinates chosen to represent the matrices), we obtain
$$\frac{|G_1|}{|H_{T_1'}|} = \frac{|G_1|}{|\{\operatorname{Id} \}|}=2,\ \frac{|G_1|}{|H_{T_2'}|} =\frac{|G_1|}{|G_1|} =1,\ \frac{|G_1|}{|H_{T_3'}|} =\frac{|G_1|}{|\{\operatorname{Id} \}|}=2.$$
Hence, $g_{0,1}(G) = \gcd(1,2)=1$, as it should be, for an elliptic curve $E/\Q$ with $2$-adic image \texttt{X235l} has a $2$-torsion point defined over $\Q$ (in fact, $E(\Q)[2^\infty]\cong \Z/8\Z$).

Let now $(s,N)=(1,5)$. Since $G$ is defined modulo $16$, the image modulo $2^5$ is the full inverse image of $G_4$ in $\GL(2,\Z/32\Z)$. We compute $G_5$, then, as the subgroup mod $32$ generated by the lift of generators of $G_4$, together with generators of the kernel of reduction $\GL(2,\Z/32\Z)\to \GL(2,\Z/16\Z)$. We obtain that $G_5$ is generated by
$$
G_5 = \left\langle
\begin{pmatrix}
 1 & 0\\
 1 & 1 
 \end{pmatrix} ,
 \begin{pmatrix}
  1 & 0\\
   12 & 1  \end{pmatrix} ,\ldots ,
 \begin{pmatrix}
 1& 0\\ 
 8 & 1 
  \end{pmatrix} ,
   \begin{pmatrix}
   17& 0\\ 
   0 & 1 
    \end{pmatrix} ,
     \begin{pmatrix}
     1& 16\\ 
     0 & 1 
      \end{pmatrix} ,
       \begin{pmatrix}
       1& 0\\ 
       16 & 1 
        \end{pmatrix} ,
         \begin{pmatrix}
         1& 0\\ 
         0 & 17 
          \end{pmatrix} 
\right\rangle,
$$
which is a subgroup of order $4096$. Now we need to parametrize subgroups $T'\subset E[5]$ such that $T'\cong \Z/2\Z\oplus \Z/2^5\Z$. Equivalently, we are parametrizing structures $T'=\langle E[2],Q\rangle$ where $Q$ is a point of exact order $2^5$. For each such choice of $Q$, we first find $H_Q$, the stabilizer of $Q$ in $G_5$. Then,
$$H_{T'} = H_Q\cap \{M \in \GL(2,\Z/2^5\Z) : M\equiv \operatorname{Id} \bmod \  2  \}.$$
For instance, let $Q=(1,0)$. Then,
$$H_Q= \left\langle
     \begin{pmatrix}
     1& 16\\ 
     0 & 1 
      \end{pmatrix} ,
       \begin{pmatrix}
       1& 0\\ 
       0 & 25 
        \end{pmatrix}  
\right\rangle,
$$
and $H_{T'}=H_Q$ because every matrix in $H_Q$ already fixes $E[2]$ as well. Thus,
$$\frac{|G_5|}{|H_{T'}|} =\frac{4096}{8} = 512.$$
Similarly, let $Q=(0,1)$. Then,
$$
H_Q = \left\langle
\begin{pmatrix}
 1 & 0\\
 1 & 1 
 \end{pmatrix} ,
 \begin{pmatrix}
  1 & 0\\
   12 & 1  \end{pmatrix} ,
 \begin{pmatrix}
 9& 0\\ 
 0 & 1 
  \end{pmatrix} ,
   \begin{pmatrix}
   1& 0\\ 
   14 & 1 
    \end{pmatrix} ,
     \begin{pmatrix}
     5& 0\\ 
     0 & 1 
      \end{pmatrix} ,
       \begin{pmatrix}
       15& 0\\ 
       0 & 1 
        \end{pmatrix} ,
         \begin{pmatrix}
         1& 0\\ 
         8 & 1 
          \end{pmatrix},\begin{pmatrix}
                 17& 0\\ 
                 0 & 1 
                  \end{pmatrix} 
\right\rangle,
$$
while  
$$
H_{T'} = \left\langle
 \begin{pmatrix}
  1 & 0\\
   12 & 1  \end{pmatrix} ,
 \begin{pmatrix}
 9& 0\\ 
 0 & 1 
  \end{pmatrix} ,
   \begin{pmatrix}
   1& 0\\ 
   14 & 1 
    \end{pmatrix} ,
     \begin{pmatrix}
     5& 0\\ 
     0 & 1 
      \end{pmatrix} ,
       \begin{pmatrix}
       15& 0\\ 
       0 & 1 
        \end{pmatrix} ,
         \begin{pmatrix}
         1& 0\\ 
         8 & 1 
          \end{pmatrix},\begin{pmatrix}
                 17& 0\\ 
                 0 & 1 
                  \end{pmatrix} 
\right\rangle,
$$
because every matrix in $H_Q$ already fixes $E[2]$ as well, except for $\begin{pmatrix}
                 1& 0\\ 
                 1 & 1 
                  \end{pmatrix}$. Thus,
$$\frac{|G_5|}{|H_{T'}|} =\frac{4096}{256} = 16.$$
It follows that $g_{1,5}(G)$ is a divisor of $16=2^4$. In fact, the calculation for all possible $T'$ reveals that $g_{1,5}(G)=2^4$, and in fact, $g_{1,5}(G^k)$ is divisible by $2^4$, for all $1\leq k\leq 1208$. Hence, $g_{1,5}(\Q,2)=2^4$ as it appears in Table \ref{ex}. Since the value $2^4$ is in fact achieved for a specific $2$-adic image that occurs over $\Q$, it follows that $g_{1,5}(\Q,2)=2^4=m_{1,5}(\Q,2)$ as well. Similarly, our calculations show that, in all cases, the values $g_{s,M}(\Q,2)$ are best possible for $0\leq s \leq 3$, i.e., there is a type of representation $G^i$ such that $g_{s,M}(\Q,2)=g_{s,M}(G^i)=m_{s,M}(G^i)$. This shows (a).

Moreover, for any $1\leq i\leq 1208$, and if $s\geq 4$, then the data computed in \texttt{2primary\_Ss.txt} shows that $g_{s,M}(G^i)$ is divisible by $2\cdot g_{s,M}(\Q,2)$ or $3\cdot g_{s,M}(\Q,2)$. And $g_{s,M}(G^i)$ is divisible by $3\cdot g_{s,M}(\Q,2)$ but not by $3\cdot g_{s,M}(\Q,2)$ only in one case, that of \texttt{X441}. This shows (b) and (c). 

Finally, it follows from (b) that the minimal values $m_{s,M}(\Q,2)$ are equal to $g_{s,M}(\Q,2)$ for $0\leq s \leq 3$, and equal to $2\cdot g_{s,M}(\Q,2)$ for $s\geq 4$. This shows (d).
\end{proof} 

We remark that the values computed for Table \ref{ex} are fairly regular. Indeed, for $1\leq s\leq 3$ and $3\leq N\leq 5$, we have $g_{s,N}(\Q,2)=2^{2(s+N-4)}$. Thus, the formulas can be simplified as follows.

\begin{cor}\label{cor-Q}
Let $E/\Q$ be an elliptic curve without CM. Let $0\leq s \leq N$ be fixed integers, and let $T\subseteq E[2^N]$ be a subgroup isomorphic to $\Z/2^s\Z \oplus \Z/2^{N}\Z$. Then, $[\Q(T):\Q]$ is divisible by $2$ if $s=N=2$, and otherwise by
$$\begin{cases}
2^{2N-7} & \text{ if } s=0 \text{ and } N\geq 4, \text{ or}\\
2^{2N+2s-8 } & \text{ if } s\geq 1 \text{ and } N\geq \max\{3,s\},\\
\end{cases}$$
unless $s\geq 4$ and $j(E)$ is one of the two values
$$-\frac{3\cdot 18249920^3}{17^{16}} \text{ or } -\frac{7\cdot 1723187806080^3}{79^{16}}$$
in which case $[\Q(T):\Q]$ is divisible by $3\cdot 2^{2N+2s-9}$. Furthermore, these divisibility properties are best possible, in the sense that for each $s,N$ there exists an elliptic curve $E/\Q$ and a subgroup $T\subseteq E[2^N]$ such that $[\Q(T):\Q]$ equals the bound above.
\end{cor}
\begin{proof}
By Corollary \ref{cor-gsN}, and specifically parts (2) and (b), it follows that $[\Q(T):\Q]$ is divisible by  
$$\begin{cases}
2^{2N-7} & \text{ if } s=0 \text{ and } N\geq 4, \text{ or}\\
2^{2N+2s-8 } & \text{ if } 1\leq s\leq 3 \text{ and } N\geq 3,\\
2^{2N+2s-8 } \text{ or } 3\cdot 2^{2N+2s-9} & \text{ if } 4\leq s \leq N.\\
\end{cases}$$
However, part (c) says that the divisibility by $3\cdot 2^{2N+2s-9}$ and not by $2^{2N+2s-8}$ only occurs when the image is \texttt{X441}. By the work of \cite{rouse}, Section 8.3, the elliptic curves with $2$-adic image \texttt{X441} are parametrized by a modular curve $\mathcal{X}_{441}$ which, in turn, is isomorphic to $X^+_{ns}(16)$,  whose noncuspidal
points classify elliptic curves whose mod 16 image of Galois is contained in the
normalizer of a non-split Cartan subgroup (see Remarks 1.4 and 7.1, and Section 8.3 in \cite{rouse}), and it is given by a model $$X^+_{ns}(16) : y^2   = x^6-3x^4+x^2+1.$$
The non-cuspidal rational points on $X^+_{ns}(16)$ were first computed by Baran \cite{baran} (and calculated again in \cite{rouse}, Section 8.3), and the only two non-CM associated $j$-invariants are 
$$-\frac{3\cdot 18249920^3}{17^{16}} \text{ or } -\frac{7\cdot 1723187806080^3}{79^{16}}.$$
Hence, Corollary \ref{cor-gsN}, part (c), implies that if $j(E)$ is not one of these two values, and $s\geq 4$, then $[\Q(T):\Q]$ is actually divisible by $2^{2N+2s-8}$, as claimed.

Finally, the assertion that the divisibilities are best possible follows from Corollary \ref{cor-gsN}, parts (a) and (b).
\end{proof}

We finish this section with a proof of Theorem \ref{thm-mainintro3} from the introduction. 

\begin{proof}[Proof of Theorem \ref{thm-mainintro3}] Let $E/\Q$ be an elliptic curve without complex multiplication, and let us first assume that $p$ is a prime satisfying (A) and (B). In particular, Theorem \ref{thm-serre2} says that the representation $\rho_{E,p}:\GQ\to \GL(2,\Z/p\Z)$ must be surjective (for $p>13$ and $p\neq 17$ or $37$). By \cite{serre2}, IV-24, Lemma 3, if $\rho_{E,p}$ is surjective and $p\geq 5$, then the $p$-adic representation $\rho_{E,p^\infty}$ must be surjective as well. Thus, Corollary \ref{cor-full} gives an exact value for $[\Q(T):\Q]$ in this case, which coincide with the bounds stated in Theorem \ref{thm-mainintro3}.

Otherwise, assume that $p$ is a prime satisfying (A) and (C). Then, $G=\rho_{E,p^\infty}(\GQ)$ is defined modulo $p$. Moreover, the mod-$p$ image is either surjective, or the image is contained in one of the maximal subgroups of type (1), (2), or (3) of Theorem \ref{thm-serre2}. In \cite{sutherland} and \cite{zywina1} such types of images have been completely classified, and there are only $63$ possibilities of non-surjective mod-$p$ images $G^i_1\cong \rho_{E,p}(\GQ)$, $1\leq i\leq 63$. For each $G^i$ defined to be the full inverse $p$-adic image of $G^i_1$, the bounds of Cor. \ref{cor-gsNgeneral} apply, where $n(K,p)$ can be taken to be $1$ because these representations are defined modulo $p$. Thus, only $g_{k,1}(G^i)$, and $m_{k,1}(G^i)$, for $k=0,1$ play a role here. Each of these constants have been calculated with Magma, and recorded in Tables \ref{exSZ11} and \ref{exSZ37} below. Finally, in Table \ref{exQp0} we have calculated and recorded $g_{k,1}(\Q,p)$ and $m_{k,1}(\Q,p)$, for $k=0,1$.
\end{proof}

\section{Examples}\label{sec-examples}
In Table \ref{ex} we have given the values of $g_{s,N}$, for $0\le s\le N\le 5$, which played an important role in the bounds given by Corollary \ref{cor-Q}, and as claimed in the statement, for each pair $s,N$ there is an elliptic curve $E_{s,N}$ over $\Q$, and  a subgroup $T_{s,N}\subseteq E_{s,N}(\overline{\Q})$ isomorphic to $\Z/2^s\Z\oplus \Z/2^{N}\Z$, such that $[\Q(T_{s,N}):\Q]$ equals the bound given by Corollary \ref{cor-Q}. In fact, there are infinitely many non-isomorphic curves over $\Q$ with this property, given by one-parameter families (except for the two exceptional $j$-invariants mentioned in the statement of Corollary \ref{cor-Q}). In particular, the elliptic curves with $2$-adic images  \texttt{X193n} and \texttt{X441} up to conjugation (following the notation of \cite{rouse}), which are parametrized by the elliptic surface  $\mathcal{X}_{193n}$ and the modular curve $\mathcal{X}_{441}$, achieve the bound in all cases except for $s=0$ and $N\geq 1$, or $s=N=2$, which are achieved by the $2$-adic images  \texttt{X235l} or \texttt{X58i} up to conjugation, respectively (and parametrized by the surfaces $\mathcal{X}_{235l}$ or $\mathcal{X}_{58i}$ which we will give below). Here are the models of the modular curves (note that we have simplified the models of $\mathcal{X}_{58i}$, $\mathcal{X}_{193n}$, and $\mathcal{X}_{235l}$ given in \cite{rouse}, by changing variables so that $(0,0)$ belongs to the elliptic surface).

\begin{align*}
\mathcal{X}_{58i} &:  y^2=x^3 - 2(t^4+1)x^2 + (t^{8}-2t^4+1) x,\\
\mathcal{X}_{193n} &:  y^2=x^3 + (256t^8 - 256t^6 + 352t^4 - 16t^2 + 
1)x^2 + 256t^4(4t^2-1)^4 x,\\
\mathcal{X}_{235l} &:  y^2= x^3 + (t^8 - 4t^6 - 2t^4 - 4t^2 + 1)x^2 + 16t^8x,\\
\mathcal{X}_{441} &: y^2   = x^6+x^4-3x^2+1.
\end{align*}

As mentioned in the proof of Corollary \ref{cor-Q}, the curve $\mathcal{X}_{441} = X^+_{ns}(16)$, a modular curve of genus $2$, contains only two non-cuspidal rational points that correspond to non-CM elliptic curves, whose $j$-invariants are given in the statement of the corollary.

In this section we show explicitly, as an example, the field of definition of $T=\Z/2^N\Z$  for $N\le 5$ (with $T$ chosen so that $[\Q(T):\Q]$ is minimal) in the elliptic surface $\mathcal{X}_{235l}$ and points of order $2^N$ for $N=3,4,5$.  The $2$-adic image \texttt{X235l} (as always, following the notation of \cite{rouse}) is defined modulo $16$. For each elliptic curve $E$ with image \texttt{X235l} (up to conjugation) there is a subgroup $T\cong \Z/2^N\Z \subseteq E[2^N]$, for $N\geq 4$, such that $[\Q(T):\Q]=2^{2N-7}$. Indeed, let us show that
$$g_{0,4}(\texttt{X235l})=g_{0,4}(\Q,2)=2,\ \text{ and }\ g_{0,5}(\texttt{X235l})=g_{0,5}(\Q,2)=2^3.$$
In order to do this, we show explicitly that the curves with image \texttt{X235l} (up to conjugation) have a $16$-torsion point defined over a quadratic extension of $\Q$, and a $32$-torsion point defined over an extension of degree $8$.  As shown in \cite{rouse}, the curves with $2$-adic image \texttt{X235l} are parametrized by the surface $\mathcal{X}_{235l}$ that appears above (see also our Remark \ref{rem-rzb} to find the model in their database).  The torsion structure and generators of $\mathcal{X}_{235l}$ over $\Q(t)$ can be computed with Magma, and it is isomorphic to $\Z/8\Z$ with the following generator:
$$
P_8=(4t^2 ,-4t^2(t^4 - 1)).
$$
Over the quadratic extension $K=\Q(\sqrt{(t^4-1)(t^2+2t-1)})$ the point:
$$
P_{16}=(2t(t^4-1)+2t\alpha,2t(t^8-4t^6+2t^5-2t^4-2t+1)+2t(t^4+2t-1)\alpha),
$$
where $\alpha=2t+(t-1)\sqrt{(t^4-1)(t^2+2t-1)}$, is of order $16$.

To obtain a point of order $32$ it is necessary to go to $L=\Q(t)(\beta)$, a degree $4$ extension of $K$, where $\beta$ satisfies:
$$
\begin{array}{l}
\beta^4 + (-12t\alpha - 2t^8 + 8t^6 - 12t^5 + 4t^4 + 8t^2 + 12t - 2)\beta^2 + ((-16t^5 - 32t^2 + 16t)\alpha - 16t^9 + 64t^7 -\\ \qquad  32t^6 + 32t^5 + 32t^2 - 16t)\beta + (-4t^9 + 16t^7 - 24t^6 + 8t^5 - 32t^3 + 24t^2 - 4t)\alpha + t^{16} - 8t^{14} - 4t^{13}\\
\qquad + 12t^{12} + 16t^{11} - 16t^{10} + 12t^9 + 22t^8 - 48t^7 + 56t^6 - 12t^5 + 12t^4 + 32t^3 - 32t^2 + 4t + 1=0
\end{array}
$$
and the point of order $32$ is:
$$
\begin{array}{rcl}
P_{32}&  = & (1/2\beta^2 - t\alpha - 1/2t^8 + 2t^6 - t^5 + t^4 + 2t^2 + t - 1/2 , \\
& & \qquad 1/2\beta^3 + (-3t\alpha - 1/2t^8 + 2t^6 - 3t^5 + t^4 + 2t^2 + 3t - 
1/2)\beta \\
& & \qquad \qquad  + (-2t^5 - 4t^2 + 2t)\alpha - 2t^9 + 8t^7 - 4t^6 + 4t^5 + 
4t^2 - 2t).
\end{array}
$$

\begin{remark}\label{rem-2div}
	In the example above we have  $2P_{32}=P_{16}$, and $2P_{16}=P_{8}$. We have halved these points using the $2$-divisibility method (cf. \cite{jeon}, or \cite[Prop. 12]{GJT14}). This method establishes that if $E$ is an elliptic curve defined over a field $K$ and $P\in E(K)[N]$ then  there exists $Q\in E(L)[2N]$ where $[L:K]\le 4$ \cite[Theorem 3.1]{jeon}. We note here that while  \cite{GJT14} or \cite{jeon} use the $2$-divisibility method over number fields, the same method applies to function fields as well.
\end{remark}

\subsection{Examples of $p$-adic representations defined modulo $p$}\label{sec-modpimage} In this section we give examples of elliptic curves $E/\Q$ and primes $p$ such that the condition (C) of Theorem \ref{thm-mainintro3} is verified, i.e., $p$-adic representations attached to elliptic curves that are defined modulo $p$. We have also calculated the constants $g_{k,1}$ and $m_{k,1}$, for $k=0,1$, for each of  the examples that appear in this section, and they are listed in Table \ref{exSZ11}.

For $p=2$, the classification of \cite{rouse} of all possible $2$-adic images (for non-CM elliptic curves) includes the level of definition of the image. In particular, the images \texttt{X1}, \texttt{X2}, \texttt{X6}, and \texttt{X8} are the only images defined modulo $2$, and these correspond to $\GL(2,\Z/2\Z)$, \texttt{2Cn}, \texttt{2B}, and \texttt{2Cs} (in the notation of \cite{sutherland}) for \texttt{X1}, \texttt{X2}, \texttt{X6}, and \texttt{X8}, respectively. The following elliptic curves are examples in each image defined modulo $2$:
\begin{itemize} 
\item $2$-adic image \texttt{X1}, corresponds to all of $\GL(2,\Z/2\Z)$ modulo $2$, and \texttt{11a1} is an example.
\item $2$-adic image \texttt{X2}, corresponds to image \texttt{2Cn} modulo $2$, and \texttt{196a1} is an example.
\item $2$-adic image \texttt{X6}, corresponds to image \texttt{2B} modulo $2$, and \texttt{69a1} is an example.
\item $2$-adic image \texttt{X8}, corresponds to image \texttt{2Cs} modulo $2$, and \texttt{315b2} is an example.
\end{itemize} 

For $p>2$, in order to check that an image is defined modulo $p$, it suffices to check that the image modulo $p^2$ is the full inverse image of the mod-$p$ image, by the following result.

\begin{lemma}[\cite{arai}, Lemma 2.2]\label{lem-arai} Let $n\geq 1$ be an integer and let $p>2$ be a prime. Let $H$ be a closed subgroup of $\GL(2,\Z_p)$. Then $H$ contains $1+p^nM(2,\Z_p)$ if and only if $H\bmod p^{n+1}$ contains $1+p^nM(2,\Z/p^{n+1}\Z)$.
\end{lemma}

We have used the previous Lemma, to find examples where $\rho_{E,9}(\GQ)$ is the full inverse image of $\rho_{E,3}(\GQ)$. In order to show this, we used Magma to verify that $$|\Gal(\Q(E[9])/\Q)|/|\Gal(\Q(E[3])/\Q)|=3^4.$$
We obtained the following examples of each of the images defined modulo $3$:
\begin{itemize}
\item Image $\GL(2,\Z/3\Z)$ modulo $3$, the curve \texttt{11a1} is an example.
\item Image \texttt{3Cs.1.1} modulo $3$, the curve \texttt{14a1} is an example.
\item Image \texttt{3Cs} modulo $3$, the curve \texttt{98a3} is an example.
\item Image \texttt{3B.1.1} modulo $3$, the curve \texttt{30a1} is an example.
\item Image \texttt{3B.1.2} modulo $3$, the curve \texttt{20a3} is an example.
\item Image \texttt{3Ns} modulo $3$, the curve \texttt{338d1} is an example.
\item Image \texttt{3B} modulo $3$, the curve \texttt{50b1} is an example.
\item Image \texttt{3Nn} modulo $3$, the curve \texttt{245a1} is an example.
\end{itemize}
Let us verify, for instance, that if $E$ is the curve with Cremona label \texttt{11a1}, then $\rho_{E,3^\infty}(\GQ)=\GL(2,\Z_3)$. Either by looking up the mod $3$ image in the Cremona database \cite{cremonaweb}, or by direct computation (using the $3$-division polynomial and checking that $[\Q(E[3]):\Q]=48=(3-1)\cdot 3\cdot (3^2-1)=|\GL(2,\Z/3\Z)|$), we verify first that $\rho_{E,3}(\GQ)=\GL(2,\Z/3\Z)$. We note here that $[\Q(x(E[3])):\Q]=24$. Now we compute $[\Q(E[9]):\Q(E[3])]$. The $9$-th division polynomial for $E/\Q$ factors as $\psi_9(x)=\psi_3(x)f(x)$ over $\Q[x]$, where the degrees of $\psi_3$ and $f$ are $4$ and $36$, respectively. Here $\psi_3(x)$ is the $3$-rd division polynomial, whose splitting field is $\Q(x(E[3]))$, and the splitting field of $f(x)$ is $\Q(x(E[9]))$. The extension $L/\Q$ generated by a root of $f(x)$, thus, is a number field of degree $36$, and Magma tells us that the Galois group of $f$ has degree $1944$. Hence, the Galois closure of $L$, i.e., $\Q(x(E[9]))$ is of degree $1944=24\cdot 3^4$ over $\Q$, and the extension $\Q(x(E[9]))/\Q(x(E[3]))$ is of degree $3^4$. Therefore $$[\Q(E[9]):\Q(E[3])]=[\Q(x(E[9])):\Q(x(E[3]))]=3^4.$$
Hence, by Lemma \ref{lem-arai}, we conclude that $\rho_{E,3^\infty}(\GQ)=\GL(2,\Z_3)$, as desired.

Unfortunately, we have not been able to verify $\rho_{E,p^2}(\GQ)$ is the full inverse image of the mod-$p$ image, for any elliptic curve with non-surjective image modulo $p>3$.

 \subsection{About the Tables \ref{exSZ11} and \ref{exSZ37}}

As mentioned in Conjecture \ref{conj-zyw}, Sutherland and Zywina (\cite{sutherland}, \cite{zywina1}) have found a list of $63$ exceptional mod-$p$ images that do occur for non-CM curves $E/\Q$, and that, together with $\GL(2,\F_p)$, would be the complete  list if the answer to Serre's uniformity question is positive. We list all such possible images in Tables \ref{exSZ11} and \ref{exSZ37}, together with the values of the prime $p$, the label of the type of mod $p$ representation, and the constants $g_{0,1}(G)$, $m_{0,1}(G)$, and $g_{1,1}(G)=m_{1,1}(G)$ that appear in Theorem \ref{thm-mainintro3}. In other words, if $E/\Q$ is an elliptic curve such that $\rho_{E,p^\infty}(\GQ)$ is the full inverse image of its mod $p$ representation, then Theorem \ref{thm-mainintro3} applies with the constants as given in these tables. In addition to the constants we have added an example of an elliptic curve whose mod $p$ representation is the group indicated in the second column. However, as explained in the previous section, for $p>3$ we have not been able to verify that, for each of the examples $E/\Q$ we give, the representation $\rho_{E,p^\infty}$ is indeed defined modulo $p$. 

Finally, we note that $m_{0,1}(G)$ and $m_{1,1}(G)$ appear listed in \cite{sutherland}, Tables 3 and 4, as $d_1$ and $d$, respectively. The tables in \cite{sutherland} also give generators for each type of image.

\vskip 4in


\begin{table}
\begin{tabular}{ccccccc}
\hline
$p$ & G & $g_{0,1}(G)$ & $m_{0,1}(G)$ & $g_{1,1}(G)=m_{1,1}(G)$  & & Example $E/\Q$\\
\hline
$2$ &  \texttt{2Cs} &  $1$ & $1$ & $1$ & & \texttt{315b2}\\
$2$ &  \texttt{2B} &  $1$ & $1$ & $2$ & & \texttt{69a1}\\
$2$ &  \texttt{2Cn} &  $3$ & $3$ & $3$ & & \texttt{196a1}\\
$2$ &  $\GL(2,\F_2)$ &  $3$ & $3$ & $6$ & & \texttt{11a1}\\
\hline
$3$ &  \texttt{3Cs.1.1} &  $1$ & $1$ & $2$ & & \texttt{14a1}\\
$3$ &  \texttt{3Cs} &  $2$ & $2$ & $4$ & & \texttt{98a3}\\
$3$ &  \texttt{3B.1.1} &  $1$ & $1$ & $6$ & & \texttt{30a1}\\
$3$ &  \texttt{3B.1.2} &  $1$ & $2$ & $6$ & & \texttt{20a3}\\
$3$ &  \texttt{3Ns} &  $4$ & $4$ & $8$ & & \texttt{338d1}\\
$3$ &  \texttt{3B} &  $2$ & $2$ & $12$ & & \texttt{50b1}\\
$3$ &  \texttt{3Nn} &  $8$ & $8$ & $16$ & & \texttt{245a1}\\
$3$ &  $\GL(2,\F_3)$ &  $8$ & $8$ & $48$ & & \texttt{11a1}\\
\hline
$5$ &  \texttt{5Cs.1.1} &  $1$ & $1$ & $4$ & & \texttt{11a1}\\
$5$ &  \texttt{5Cs.1.3} &  $2$ & $2$ & $4$ & & \texttt{275b2}\\
$5$ &  \texttt{5Cs.4.1} &  $2$ & $2$ & $8$ & & \texttt{99d2}\\
$5$ &  \texttt{5Ns.2.1} &  $8$ & $8$ & $16$ & & \texttt{6975a1}\\
$5$ &  \texttt{5Cs} &  $4$ & $4$ & $16$ & & \texttt{18176b2}\\
$5$ &  \texttt{5B.1.1} &  $1$ & $1$ & $20$ & & \texttt{11a3}\\
$5$ &  \texttt{5B.1.2} &  $1$ & $4$ & $20$ & & \texttt{11a2}\\
$5$ &  \texttt{5B.1.3} &  $2$ & $4$ & $20$ & & \texttt{50a1}\\
$5$ &  \texttt{5B.1.4} &  $2$ & $2$ & $20$ & &  \texttt{50a3}\\
$5$ &  \texttt{5Ns} &  $8$ & $8$ & $32$ & &  \texttt{608b1}\\
$5$ &  \texttt{5B.4.1} &  $2$ & $2$ & $40$ & &  \texttt{99d1}\\
$5$ &  \texttt{5B.4.2} &  $2$ & $4$ & $40$ & &  \texttt{99d3}\\
$5$ &  \texttt{5Nn} &  $24$ & $24$ & $48$ & &  \texttt{675b1}\\
$5$ &  \texttt{5B} &  $4$ & $4$ & $80$ & &  \texttt{338d1}\\
$5$ &  \texttt{5S4} &  $24$ & $24$ & $96$ & &  \texttt{324b1}\\
$5$ &  $\GL(2,\F_5)$ &  $24$ & $24$ & $480$ & & \texttt{14a1}\\
\hline
$7$ &  \texttt{7Ns.2.1} &  $3$ & $6$ & $18$ & &  \texttt{2450ba1}\\
$7$ &  \texttt{7Ns.3.1} &  $6$ & $12$ & $36$ & &  \texttt{2450a1}\\
$7$ &  \texttt{7B.1.1} &  $1$ & $1$ & $42$ & &  \texttt{26b1}\\
$7$ &  \texttt{7B.1.2} &  $3$ & $3$ & $42$ & &  \texttt{637a1}\\
$7$ &  \texttt{7B.1.3} &  $1$ & $6$ & $42$ & &  \texttt{26b2}\\
$7$ &  \texttt{7B.1.4} &  $1$ & $3$ & $42$ & &  \texttt{294a1}\\
$7$ &  \texttt{7B.1.5} &  $3$ & $6$ & $42$ & &  \texttt{637a2}\\
$7$ &  \texttt{7B.1.6} &  $1$ & $2$ & $42$ & &  \texttt{294a2}\\
$7$ &  \texttt{7Ns} &  $12$ & $12$ & $72$ & &  \texttt{9225a1}\\
$7$ &  \texttt{7B.6.1} &  $2$ & $2$ & $84$ & &  \texttt{208d1}\\
$7$ &  \texttt{7B.6.2} &  $6$ & $6$ & $84$ & &  \texttt{5733d1}\\
$7$ &  \texttt{7B.6.3} &  $2$ & $6$ & $84$ & &  \texttt{208d2}\\
$7$ &  \texttt{7Nn} &  $48$ & $48$ & $96$ & &  \texttt{15341a1}\\
$7$ &  \texttt{7B.2.1} &  $3$ & $3$ & $126$ & &  \texttt{162b1}\\
$7$ &  \texttt{7B.2.3} &  $3$ & $6$ & $126$ & &  \texttt{162b3}\\
$7$ &  \texttt{7B} &  $6$ & $6$ & $252$ & &  \texttt{162c1}\\
$7$ &  $\GL(2,\F_7)$ &  $48$ & $48$ & $2016$ & & \texttt{11a1}\\
\hline
\end{tabular}
\caption{$g_{k,1}(G)$ and $m_{k,1}(G)$, for $k=0,1$, and example curves.}\label{exSZ11}
\end{table}

\begin{table}
\renewcommand{\arraystretch}{1.1}
\begin{tabular}{ccccccc}
\hline
$p$ & $G$ & $g_{0,1}(G)$ & $m_{0,1}(G)$ & $g_{1,1}(G)=m_{1,1}(G)$ & & Example $E/\Q$\\
\hline
$11$ &  \texttt{11B.1.4} &  $5$ & $5$ & $110$ & &  \texttt{121a2}\\
$11$ &  \texttt{11B.1.5} &  $5$ & $5$ & $110$ & &  \texttt{121c2}\\
$11$ &  \texttt{11B.1.6} &  $5$ & $10$ & $110$ & &  \texttt{121a1}\\
$11$ &  \texttt{11B.1.7} &  $5$ & $10$ & $110$ & &  \texttt{121c1}\\
$11$ &  \texttt{11B.10.4} &  $10$ & $10$ & $220$ & &  \texttt{1089f2}\\
$11$ &  \texttt{11B.10.5} &  $10$ & $10$ & $220$ & &  \texttt{1089f1}\\
$11$ &  \texttt{11Nn} &  $120$ & $120$ & $240$ & &  \texttt{232544f1}\\
$11$ &  $\GL(2,\F_{11})$ &  $120$ & $120$ & $13200$ & & \texttt{11a1}\\
\hline
$13$ &  \texttt{13S4} &  $24$ & $72$ & $288$ & &  \texttt{152100g1}\\
$13$ &  \texttt{13B.3.1} &  $3$ & $3$ & $468$ & &  \texttt{147b1}\\
$13$ &  \texttt{13B.3.2} &  $3$ & $12$ & $468$ & &  \texttt{147b2}\\
$13$ &  \texttt{13B.3.4} &  $6$ & $6$ & $468$ & &  \texttt{24843o1}\\
$13$ &  \texttt{13B.3.7} &  $6$ & $12$ & $468$ & &  \texttt{24843o2}\\
$13$ &  \texttt{13B.5.1} &  $4$ & $4$ & $624$ & &  \texttt{2890d1}\\
$13$ &  \texttt{13B.5.2} &  $4$ & $12$ & $624$ & &  \texttt{2890d2}\\
$13$ &  \texttt{13B.5.4} &  $12$ & $12$ & $624$ & &  \texttt{216320i1}\\
$13$ &  \texttt{13B.4.1} &  $6$ & $6$ & $936$ & &  \texttt{147c1}\\
$13$ &  \texttt{13B.4.2} &  $6$ & $12$ & $936$ & &  \texttt{147c2}\\
$13$ &  \texttt{13B} &  $12$ & $12$ & $1872$ & &  \texttt{2450l1}\\
$13$ &  $\GL(2,\F_{13})$ &  $168$ & $168$ & $26208$ & & \texttt{11a1}\\
\hline
$17$ &  \texttt{17B.4.2} &  $8$ & $8$ & $1088$ & &  \texttt{14450n1}\\
$17$ &  \texttt{17B.4.6} &  $8$ & $16$ & $1088$ & &  \texttt{14450n2}\\
$17$ &  $\GL(2,\F_{17})$ &  $288$ & $288$ & $78336$ & & \texttt{11a1}\\
\hline
$37$ &  \texttt{37B.8.1} &  $12$ & $12$ & $15984$ & &  \texttt{1225e1}\\
$37$ &  \texttt{37B.8.2} &  $12$ & $36$ & $15984$ & &  \texttt{1225e2}\\
$37$ &  $\GL(2,\F_{37})$ &  $1368$ & $1368$ & $1822176$ & & \texttt{11a1}\\
\hline
else &  $\GL(2,\F_{p})$ &  $p^2-1$ & $p^2-1$ & $(p-1)p(p^2-1)$ & & \texttt{11a1}\\
\hline
\end{tabular}
\caption{$g_{k,1}(G)$ and $m_{k,1}(G)$, for $k=0,1$, and example curves.}\label{exSZ37}
\end{table}

\eject


\begin{thebibliography}{9}
%
\bibitem{arai} K. Arai, {\it On uniform lower bound of the Galois images associated to elliptic curves},  J. Th\'eor. Nombres Bordeaux, Tome 20, n. 1 (2008), pp. 23-43.

\bibitem{baran} B. Baran, {\em Normalizers of non-split Cartan subgroups, modular curves, and the class number one
problem}, J. Number Theory 130 (2010), pp. 2753-2772.

\bibitem{bilu} Y. Bilu, P. Parent, {\em Serre's uniformity problem in the split Cartan case}, Ann. of Math., Vol. 173 (2011), Issue 1, pp. 569-584.

\bibitem{bilu2} Y. Bilu, P. Parent, M. Rebolledo, {\it Rational points on $X_0^+(p^r)$}, Ann. Inst. Fourier, Vol. 63 n. 3 (2013), pp. 957-984.



\bibitem{magma} W. Bosma, J. Cannon, and C. Playoust, {\it The Magma algebra system. I. The user language}, J. Symbolic Comput., 24 (1997), pp. 235-265.


\bibitem{bourdon} A. Bourdon, P. L. Clark, and P. Pollack, {\it Anatomy of Torsion in the CM Case}, preprint (arXiv:1506.00565).


\bibitem{cremonaweb} J. E. Cremona, {\it Elliptic curve data for conductors up to 350.000,} 2015. Available at  \href{http://homepages.warwick.ac.uk/~masgaj/ftp/data/}{http://homepages.warwick.ac.uk/$\sim$masgaj/ftp/data/}

\bibitem{GJT14} E.~Gonz\'alez--Jim\'enez, J.M.~Tornero. {\em Torsion of rational elliptic curves over quadratic fields II.} Rev. R. Acad. Cienc. Exactas F\'{\i}s. Nat. Ser. A Math. RACSAM, to appear (arXiv:1411.3468).


\bibitem{jeon}
D. Jeon, C.H. Kim, Y. Lee. {\em Infinite families of elliptic curves over dihedral quartic number fields.} J. Number Theory 133 (2013), pp. 115-122.

\bibitem{kami} S. Kamienny, {\em  Torsion points on elliptic curves and q-coefficients of modular forms}, Invent. Math. 109 (1992), pp. 129-133.


\bibitem{kenku2} M. A. Kenku, F. Momose, {\em Torsion points on elliptic curves defined over quadratic fields}, Nagoya Math. J. 109 (1988), pp. 125-149.


\bibitem{lozano0} \'A. Lozano-Robledo, {\it On the field of definition of $p$-torsion points on elliptic curves over the rationals}, Math. Ann., Vol. 357 (2013), Issue 1, pp. 279-305.



\bibitem{mazur1} B. Mazur, {\it Rational isogenies of prime degree}, Invent. Math. 44 (1978), pp. 129 - 162.


\bibitem{najman} F. Najman, {\it Torsion of rational elliptic curves over cubic fields and sporadic points on $X_1(n)$}, Math. Res. Lett., to appear (arXiv:1211.2188).

\bibitem{prasad} D. Prasad, C. S. Yogananda, {\em Bounding the torsion in CM elliptic curves}, C. R. Math. Acad. Sci. Soc. R. Can. 23 (2001), pp. 1-5.


\bibitem{rouse} J. Rouse, D. Zureick-Brown, {\em Elliptic curves over $\Q$ and $2$-adic images of Galois}, Research in Number Theory, 1:12 (2015), data files and subgroup descriptions available at \href{http://users.wfu.edu/rouseja/2adic/}{http://users.wfu.edu/rouseja/2adic/}.


\bibitem{serre1} J-P. Serre, {\em Propri\'et\'es galoisiennes des points d'ordre fini des courbes elliptiques}, Invent. Math. 15 (1972), pp. 259-331.

\bibitem{serre2} J-P. Serre, {\em Quelques applications du th\'eor\`eme de densit\'e de Chebotarev}, Publ. Math. IHES 54 (1981), pp. 123-201.


\bibitem{silverberg} A. Silverberg, {\em Torsion points on abelian varieties of CM-type}. Compos. Math. 68 (1988), n. 3, pp. 241-249.


\bibitem{sutherland} A. Sutherland, {\em Computing images of Galois representations attached to elliptic curves}, preprint (arXiv:1504.07618).

\bibitem{zywina1} D. Zywina, {\em On the possible images of the mod $\ell$ representations associated to elliptic curves over $\Q$}, preprint (arXiv:1508.07660).

\end{thebibliography}
\end{document}